\documentclass[12pt,reqno, draft]{amsart}

\usepackage[utf8]{inputenc}
\usepackage{amsmath}
\usepackage{amsthm}
\usepackage{ulem}
\usepackage{enumitem}

\makeatletter
\newcommand{\myitem}[1][]{
  \protected@edef\@currentlabel{#1}%
\item[#1]
}
\makeatother

\DeclareMathOperator{\modf}{mod}
\DeclareMathOperator{\im}{Im}

\newcommand{\HomT}{\T}

\newcommand{\D}{\mathscr{D}}

\newcommand{\T}{\mathscr{T}}

\usepackage{tikz}
\usetikzlibrary{arrows,decorations.pathmorphing,decorations.pathreplacing, decorations.markings, matrix, patterns} 

\usepackage{mathdots}

\usepackage{color}

\usepackage{pdfsync}

\usepackage{enumitem}

\usepackage{wasysym}

\usepackage{amssymb}

\usepackage{mathrsfs}

\DeclareMathAlphabet{\mathpzc}{OT1}{pzc}{m}{it}

\usepackage[all]{xy}

\usepackage{tikz}
\usetikzlibrary{arrows,matrix,decorations.pathmorphing,decorations.pathreplacing,positioning,shapes.geometric,shapes.misc,decorations.markings,decorations.fractals,calc,patterns}

\usepackage{graphicx}

\usepackage{float}

\usepackage[bottom]{footmisc}

\usepackage{moreenum}

\usepackage{makecell}

\entrymodifiers={+!!<0pt,\fontdimen22\textfont2>}

\def\cC{\mathscr{C}}
\def\cD{\mathscr{D}}

\def\cO{\mathscr{O}}

\def\cT{\mathscr{T}}

\def\add{\operatorname{add}}

\def\adots{\mathinner{\mkern1mu\raise1.0pt\vbox{\kern7.0pt\hbox{.}}\mkern2mu\raise4.0pt\hbox{.}\mkern2mu\raise7.0pt\hbox{.}\mkern1mu}}

\def\D{\cD}

\def\dddots{\mathinner{\mkern1mu\raise10.0pt\vbox{\kern7.0pt\hbox{.}}\mkern2mu\raise5.3pt\hbox{.}\mkern2mu\raise1.0pt\hbox{.}\mkern1mu}}
\def\dddotssmall{\mathinner{\mkern1mu\raise7.0pt\vbox{\kern7.0pt\hbox{.}}\mkern-1mu\raise4pt\hbox{.}\mkern-1mu\raise1.0pt\hbox{.}\mkern1mu}}

\def\dim{\operatorname{dim}}

\def\dual{\operatorname{D}}
\def\End{\operatorname{End}}

\def\Ext{\operatorname{Ext}}

\def\Hom{\operatorname{Hom}}

\def\inj{\operatorname{inj}}

\def\Iso{\operatorname{Iso}}
\def\iso{\operatorname{iso}}

\def\mod{\operatorname{mod}}

\def\proj{\operatorname{proj}}

\def\PSL2{\operatorname{PSL}_2}

\def\SL2{\operatorname{SL}_2}

\numberwithin{equation}{section}

\newtheorem{Lemma}{Lemma}[section]
\newtheorem{Theorem}[Lemma]{Theorem}
\newtheorem{Proposition}[Lemma]{Proposition}
\newtheorem{Corollary}[Lemma]{Corollary}

\theoremstyle{definition}
\newtheorem{Definition}[Lemma]{Definition}
\newtheorem{Setup}[Lemma]{Setup}

\newtheorem{Remark}[Lemma]{Remark}

\newtheorem*{bfhpg*}{}

  {\begin{list}{}{%
    \settowidth{\labelwidth}{\textbf{#1:}}%
    \setlength{\leftmargin}{\labelwidth}\addtolength{\leftmargin}{\labelsep}}}%
  {\end{list}}

\begin{document}

\normalem
\setlength{\parindent}{0pt}
\setlength{\parskip}{7pt}

\title{Maximal $\tau_d$-rigid pairs}

\author{Karin M.\ Jacobsen}

\address{Jacobsen: Norwegian University of Science and Technology, Department of Mathematical Sciences, Sentralbygg 2, Gl\o shaugen, 7491 Trondheim, Norway}
\email{kjacobsen@math.uni-bielefeld.de}

\curraddr{Fakult\"{a}t f\"{u}r Mathematik, Universit\"{a}t Bielefeld, 33501 Bielefeld, Germany}

\author{Peter J\o rgensen}
\address{J\o rgensen: School of Mathematics and Statistics,
Newcastle University, Newcastle upon Tyne NE1 7RU, United Kingdom}
\email{peter.jorgensen@ncl.ac.uk}
\urladdr{http://www.staff.ncl.ac.uk/peter.jorgensen}


\keywords{$d$-abelian category, $( d+2 )$-angulated category, higher homological algebra, maximal $d$-rigid object, maximal $\tau_d$-rigid pair}  

\subjclass[2010]{16G10, 18E10, 18E30}


\begin{abstract} 

Let $\cT$ be a $2$-Calabi--Yau triangulated category, $T$ a cluster tilting object with endomorphism algebra $\Gamma$.  Consider the functor $\cT( T,- ) : \cT \rightarrow \mod \Gamma$.  It induces a bijection from the isomorphism classes of cluster tilting objects to the isomorphism classes of support $\tau$-tilting pairs.  This is due to Adachi, Iyama, and Reiten.

\medskip
\noindent
The notion of $( d+2 )$-angulated categories is a higher analogue of triangulated categories.  We show a higher analogue of the above result, based on the notion of maximal $\tau_d$-rigid pairs.  

\end{abstract}

\maketitle

\setcounter{section}{-1}
\section{Introduction}
\label{sec:introduction}

In triangulated categories, the notions of {\em cluster tilting objects} (introduced in \cite[p.\ 583]{BMRRT}) and {\em maximal rigid objects} have recently been extensively investigated.  They frequently coincide, by \cite[thm.\ 2.6]{ZZ}, and they are closely linked to the notion of {\em support $\tau$-tilting pairs} in abelian categories (introduced in \cite[def.\ 0.3]{AIR}).  Indeed, there is often a bijection between the cluster tilting objects in a triangulated category and the support $\tau$-tilting pairs in a suitable (abelian) module category, see \cite[thm.\ 4.1]{AIR}.

This paper investigates the analogous theory in $( d+2 )$-angulated and $d$-abelian categories, which are the main objects of higher homological algebra, see \cite[def.\ 2.1]{GKO} and \cite[def.\ 3.1]{Jasso}.  Several key properties from the classic case do not carry over. For example, cluster tilting objects are maximal $d$-rigid, but the converse is rarely true.  Moreover, the higher analogue of support $\tau$-rigid pairs permit a bijection to the maximal $d$-rigid objects, but not to the cluster tilting objects.

For further reading in higher homological algebra a number of references have been included in the bibliography, see \cite{BT}, \cite{F1}, \cite{F2}, \cite{GKO}, \cite{HI1}, \cite{HI2}, \cite{Iyama}, \cite{Iyama2}, \cite{IO2}, \cite{JJ}, \cite{Jasso}, \cite{JKu1}, \cite{JKu2}, \cite{JKv}, \cite{McM}, \cite{M}, \cite{OT}.

Let $k$ be an algebraically closed field, $d \geqslant 1$ an integer, \(\T\) a $k$-linear $\Hom$-finite \((d+2)\)-angulated category with split idempotents, see \cite[def.\ 2.1]{GKO}.  Assume that $\cT$ is $2d$-Calabi--Yau, see \cite[def.\ 5.2]{OT}, and let $\Sigma^d$ denote the \(d\)-suspension functor of $\cT$.

{\bf Cluster tilting and maximal $d$-rigid objects. }  An object $X \in \cT$ is {\em $d$-rigid} if $\Ext^d_{ \cT }( X,X ) = 0$.    We recall three important definitions.

\begin{Definition}[{\cite[def.\ 5.3]{OT}}]
\label{def:clusterTilting}
An object \(X\in\T\) is {\em Oppermann--Thomas cluster tilting in $\cT$} if:
\begin{enumerate}
\setlength\itemsep{4pt}

  \item $X$ is $d$-rigid.

  \item For any \(Y\in\T\) there exists a \((d+2)\)-angle 
\[
  X_d\rightarrow \cdots \rightarrow X_0\rightarrow Y\rightarrow \Sigma^d X_d
\]
with \(X_i\in\add X\) for all \(0\leq i\leq d\).
\end{enumerate}
\end{Definition}

\begin{Definition}
\label{def:self-orthogonal}
An object $X \in \cT$ is {\em $d$-self-perpendicular in $\cT$} if
\[
  \add X = \{\, Y \in \cT \mid \Ext^d_{ \cT }( X,Y ) = 0 \,\}.
\]
\end{Definition}

\begin{Definition}
\label{def:maxRigid}
An object \(X\in\T\) is \emph{maximal \(d\)-rigid in $\cT$} if
\[
  \add X = \{\, Y \in \T \mid \Ext^d_\T( X\oplus Y,X\oplus Y )=0 \,\}.
\]
\end{Definition}

Our first main result is:

{\bf Theorem A. }
{\em
$X$ is Oppermann--Thomas cluster tilting $\Rightarrow$ $X$ is $d$-self-perpendicular $\Rightarrow$ $X$ is maximal $d$-rigid.
}

We prove this in Theorem \ref{thm:cluster_tilting_v_maximal_rigid}.  Of equal importance is that the implications cannot be reversed in general, see Remark \ref{rmk:cluster_tilting_v_maximal_rigid}.  In particular, when $d \geqslant 2$, the class of maximal $d$-rigid objects is typically strictly larger than the class of Oppermann--Thomas cluster tilting objects, in contrast to the classic case $d = 1$ where the two classes usually coincide, see \cite[thm.\ 2.6]{ZZ}.

{\bf Maximal $\tau_d$-rigid pairs}.  Let $T \in \cT$ be an Oppermann--Thomas cluster tilting object and let $\Gamma = \End_{ \cT }( T )$.  Recall the following result.

\begin{Theorem}[{\cite[thm.\ 0.6]{JJ}}]
\label{thm:JJ1main}
Consider the essential image $\cD$  of the functor $\cT( T,- ) : \cT \rightarrow \mod \Gamma$.  Then $\cD$ is a $d$-cluster tilting subcategory of $\mod \Gamma$.  There is a commutative diagram, as shown below, where the vertical arrow is the quotient functor and the diagonal arrow is an equivalence of categories:

\centering
\begin{tikzpicture}[scale=2.5]
\node (T) at (0,1) {\(\T\)};
\node (Tadd) at (0,0) {\(\T/\add\Sigma^d T\).};
\node (D) at (1,1) {\(\D\)};
\draw[->>] (T)--node[left]{$\scriptstyle \overline{ ( - ) }$} (Tadd);
\draw[->] (T)--node[above]{\(\scriptstyle{\HomT(T,-)}\)} (D);
\draw[->] (Tadd)-- node[above, sloped] {\(\sim\)} (D);
\end{tikzpicture}
\end{Theorem}

The category $\cD$ is a $d$-abelian category by \cite[thm.\ 3.16]{Jasso}.  It has a $d$-Auslander--Reiten translation $\tau_d$, which is a higher analogue of the classic Auslander--Reiten translation $\tau$, see \cite[sec.\ 1.4.1]{Iyama2}.  A module \(M\in \D\) is called \emph{\(\tau_d\)-rigid} if \(\Hom_\Gamma(M,\tau_d M)=0\).

\begin{Remark}\label{rem:add-proj}
The classic \(\add\)-\(\proj\)-correspondence holds, as \(\HomT(T,-)\) restricts to an equivalence \(\add T\rightarrow \proj \Gamma\) . The functor also restricts to an equivalence \(\add ST\rightarrow \inj \Gamma\). \cite[lem.\ 2.1]{JJ}
\end{Remark}

It is natural to ask if $\cD$ permits a higher analogue of the $\tau$-tilting theory of \cite{AIR}.  We will not answer this question, but will instead introduce the following definitions inspired by it.

\begin{Definition}
\label{def:tau_d-rigid_pair}
A pair $( M,P )$ with $M \in \cD$ and $P \in \proj \Gamma$ is called a {\em $\tau_d$-rigid pair in $\cD$} if $M$ is $\tau_d$-rigid and $\Hom_{ \Gamma }( P,M ) = 0$. 
\end{Definition}

\begin{Definition}
\label{def:maximal_tau_d-rigid_pair}
A pair $( M,P )$ with $M \in \cD$ and $P \in \proj \Gamma$ is called a {\em maximal $\tau_d$-rigid pair in $\cD$} if it satisfies:
\begin{enumerate}
\item If \(N\in \D\) then
\[
  N \in \add M
  \Leftrightarrow
  \left\{
    \begin{array}{l}
      \Hom_\Gamma(M,\tau_d N) = 0, \\[1mm]
      \Hom_\Gamma(N,\tau_d M) = 0,\\[1mm]
      \Hom_\Gamma(P,N) = 0.
    \end{array}
  \right.
\]

\item If \(Q\in \proj\Gamma\), then
\[
  Q\in \add P \Leftrightarrow \Hom_\Gamma(Q,M)=0.
\]

\end{enumerate}
\end{Definition}

A maximal $\tau_d$-rigid pair is a $\tau_d$-rigid pair.  

Our second main result is:

{\bf Theorem B. }
{\em 
If each indecomposable object of $\cT$ is $d$-rigid, then there is a bijection
\[  
  \left\{
    \begin{array}{cc}
      \mbox{isomorphism classes of } \\
      \mbox{maximal $d$-rigid objects in $\cT$}
    \end{array}
  \right\}
  \rightarrow
  \left\{
    \begin{array}{cc}
      \mbox{isomorphism classes of} \\
      \mbox{maximal $\tau_d$-rigid pairs in $\cD$}       
    \end{array}
  \right\}.
\]
}

We prove this in Section \ref{sec:tau-TO}.  If $d = 1$, then $( M,P )$ is a maximal $\tau_1$-rigid pair if and only if it is a support $\tau$-tilting pair in the sense of \cite[def.\ 0.3(b)]{AIR}, see \cite[def.\ 0.3, prop.\ 2.3, and cor.\ 2.13]{AIR}.  Hence Theorem B is a higher analogue of the bijection
\[  
  \left\{
    \begin{array}{cc}
      \mbox{isomorphism classes of } \\
      \mbox{cluster tilting object in $\cT$}
    \end{array}
  \right\}
  \rightarrow
  \left\{
    \begin{array}{cc}
      \mbox{isomorphism classes of} \\
      \mbox{support $\tau$-tilting pairs in $\mod \Gamma$}
    \end{array}
  \right\}
\]
which exists by \cite[thm.\ 4.1]{AIR} when $\cT$ is triangulated, i.e.\ in the case $d = 1$.  However, when $d \geqslant 2$, we do not think of maximal $\tau_d$-rigid pairs as support $\tau_d$-tilting pairs.  The reason is that by Theorem B, maximal $\tau_d$-rigid pairs are linked to maximal $d$-rigid objects in higher angulated categories.  As remarked above, this class is typically strictly larger than the class of Oppermann--Thomas cluster tilting objects when $d \geqslant 2$.

Note that \cite{McM} makes an approach to higher support tilting theory.

This paper is organised as follows: Section \ref{sec:A} proves Theorem A, 
Section \ref{sec:AR-theory} investigates the precise relation between $\Hom$ spaces in $\cT$ and $\cD$, Section \ref{sec:tau-TO} proves Theorem B, and Section \ref{sec:example} gives an example.

\begin{Setup}
\label{set:blanket}
Throughout the paper we use the following notation:
\begin{description}[align=right, labelwidth=1cm]
\item[\(k\)] An algebraically closed field.
\item[\(\dual\)] The duality functor \(\Hom_k(-,k)\).
\item[\(\T\)] A \(k\)-linear, $\Hom$-finite, \((d+2)\)-angulated category with split idempotents.  We assume that $\cT$ is \(2d\)-Calabi--Yau, that is $\cT( X,Y ) \cong \dual\!\cT( Y,\Sigma^{ 2d }X )$ naturally in $X,Y \in \cT$.
\item[\( \Sigma^d \)] The \(d\)-suspension functor on \(\T\).
\item[\(T\)] An Oppermann--Thomas cluster tilting object in \(\T\).
\item[{$\overline{( - )}$}] The canonical functor $\cT \rightarrow \cT / \add \Sigma^d T$, whose target is the naive quotient category of $\cT$ modulo the morphisms which factor through an object in $\add \Sigma^d T$.
\item[\(\Gamma\)] The endomorphism ring \(\End_\T(T)\).
\item[\(\nu_\Gamma\)] The Nakayama functor on \(\modf \Gamma\).
\item[$\tau_d$] The $d$-Auslander--Reiten translation on $\mod \Gamma$.  
\item[\(\D\)] The essential image of the functor \(\HomT(T,-):\T\rightarrow \modf\Gamma\).
\end{description}
\end{Setup}

\section{Proof of Theorem A}
\label{sec:A}

\begin{Theorem}
\label{thm:cluster_tilting_v_maximal_rigid}
Let $X \in \cT$ be given.  
\begin{enumerate}
\setlength\itemsep{4pt}

  \item  There are implications
\[
  \begin{array}{c}
    \mbox{$X$ is Oppermann--Thomas cluster tilting} \\
    \Downarrow \\
    \mbox{$X$ is $d$-self-perpendicular} \\
    \Downarrow \\ 
    \mbox{$X$ is maximal $d$-rigid} \\
    \Downarrow \\
    \mbox{$X$ is $d$-rigid.}
  \end{array}
\]

  \item  If each indecomposable object in $\cT$ is $d$-rigid, then 
\[
  \mbox{ $X$ is $d$-self-perpendicular $\Leftrightarrow$ $X$ is maximal $d$-rigid. }
\]

\end{enumerate}
\end{Theorem}

\begin{proof}
(i), the first implication: Suppose $X$ is Oppermann--Thomas cluster tilting.  We must prove the equality in Definition \ref{def:self-orthogonal}, and the inclusion  $\subseteq$ is clear.  For the inclusion $\supseteq$, suppose $\Ext^d_{ \cT }( X,Y ) = 0$.  Then each morphism $X_0 \rightarrow \Sigma^d Y$ with $X_0 \in \add X$ is zero.  This applies in particular to the $( d+2 )$-angle $X_d \rightarrow \cdots \rightarrow X_0 \rightarrow \Sigma^d Y \rightarrow \Sigma^d X_d$ with $X_i \in \add X$, which exists since $X$ is Oppermann--Thomas cluster tilting.  But then the morphism $\Sigma^d Y \rightarrow \Sigma^d X_d$ is a split monomorphism, and applying $\Sigma^{ -d }$ gives a split monomorphism $Y \rightarrow X_d$ proving $Y \in \add X$.  

(i), the second implication:  Suppose that $X$ is $d$-self-perpendicular.  We must prove the equality in Definition \ref{def:maxRigid}, and the inclusion $\subseteq$ is clear.  For the inclusion $\supseteq$, suppose \(\Ext^d_\T( X \oplus Y,X \oplus Y)=0\).  Then in particular, \(\Ext^d_\T(X,Y)=0\), whence \(Y\in \add X\).

(i), the third implication: This is clear.

(ii): Suppose that each indecomposable object in $\cT$ is $d$-rigid.
Because of part (i), it is enough to prove the implication $\Leftarrow$ in (ii), so suppose that $X$ is maximal $d$-rigid.  We must prove the equality in Definition \ref{def:self-orthogonal}, and $\subseteq$ is clear.

For the inclusion $\supseteq$, observe that $\{\, Y \in \cT \mid \Ext^d_{ \cT }( X,Y ) = 0 \,\}$ is closed under direct sums and summands by additivity of \(\Ext\).  Hence it is enough to suppose that $Y$ is an indecomposable object in this set and prove $Y \in \add X$.  However, $\Ext^d_{ \cT }( X,Y ) = 0$ implies $\Ext^d_{ \cT }( Y,X ) = 0$ because $\cT$ is $2d$-Calabi--Yau, and $\Ext^d_{ \cT }( Y,Y ) = 0$ by assumption.  Finally, $X$ is $d$-rigid by part (i), so $\Ext^d_{ \cT }( X,X ) = 0$.  Combining these equalities shows $\Ext^d_{ \cT }( X \oplus Y,X \oplus Y ) = 0$, and $Y \in \add X$ follows. 
\end{proof}

\begin{Remark}
\label{rmk:cluster_tilting_v_maximal_rigid}
The implications in Theorem \ref{thm:cluster_tilting_v_maximal_rigid}(i) cannot be reversed in general:
\begin{itemize}[label=\textendash]
\item An example of a $d$-self-perpendicular object $X$ which is not Oppermann--Thomas cluster tilting is given in Section \ref{sec:example}.  In fact, the objects in the last three rows of Figure \ref{fig:example} are such examples. 
The example was originally given in \cite[p.\ 1735]{OT}.

\item An example of a maximal $d$-rigid object which is not $d$-self-perpendicular can be obtained by combining proposition 2.6 and corollary 2.7 in \cite{BMV}.  These results give a maximal $1$-rigid object which is not cluster tilting, but in the triangulated setting of \cite{BMV}, cluster tilting is equivalent to $1$-self-perpendicular, see \cite[bottom of p.\ 963]{BMV}.  

\item Finally, an example of a $d$-rigid object which is not maximal $d$-rigid is the zero object, as soon as $\cT$ has a non-zero $d$-rigid object.
\end{itemize}
\end{Remark}

We end the section by observing that Theorem \ref{thm:cluster_tilting_v_maximal_rigid}(ii) can be applied to an important class of categories.

\begin{Proposition}
\label{pro:indecomposables_are_rigid}
Let $\Lambda$ be a $d$-representation finite algebra, $\cO_{ \Lambda }$ the $( d+2 )$-angulated cluster category associated to $\Lambda$ in \cite[thm.\ 5.2]{OT}.  Then each $X \in \cO_{ \Lambda }$ satisfies 
\[
  \mbox{ $X$ is $d$-self-perpendicular $\Leftrightarrow$ $X$ is maximal $d$-rigid. }
\]
\end{Proposition}

\begin{proof}
Each indecomposable in  $\cO_{ \Lambda }$ is $d$-rigid by \cite[Lemma 5.41]{OT}, so the equivalence follows from Theorem \ref{thm:cluster_tilting_v_maximal_rigid}(ii).  
\end{proof}

\section{A dimension formula for $\Ext^d_{ \cT }$}
\label{sec:AR-theory}

Recall from Setup \ref{set:blanket} that $T$ is a fixed Oppermann--Thomas cluster tilting object in $\cT$, and that $\cT$ is $2d$-Calabi--Yau, that is, $\cT( X,Y ) \cong \dual\!\cT( Y,\Sigma^{ 2d }X )$ naturally in $X,Y \in \cT$.

\begin{Lemma}\label{lem:Nakayama}
There is a natural isomorphism
\[
  \nu_\Gamma\HomT(T,T') \cong \HomT \big( T,\Sigma^{ 2d }(T') \big)
\]
for \(T'\in\add T\).
\end{Lemma}
\begin{proof}
By the 2\(d\)-Calabi-Yau property we have
\[
  \HomT \big( T,\Sigma^{ 2d }(T') \big) \cong \dual\!\HomT(T',T).
\]
By \cite[Lemma 2.2(i)]{JJ},
\[
  \dual\!\HomT(T',T)\cong \dual\!\Hom_\Gamma \big( \HomT(T,T'),\HomT(T,T) \big) = \dual\!\Hom_\Gamma \big( \HomT(T,T'),\Gamma \big).
\]
Finally, by definition we have
\[
  \dual\!\Hom_\Gamma \big( \HomT(T,T'),\Gamma \big) = \nu_\Gamma\HomT(T,T'),
\]
see \cite[def.\ III.2.8]{bluebook1}.
\end{proof}

\begin{Lemma}
\label{lem:resolution}
If $X \in \cT$ has no non-zero direct summands in $\add \Sigma^d T$, then there exists a $( d+2 )$-angle
\[
  T_d \rightarrow \cdots \rightarrow T_0 \rightarrow X \rightarrow \Sigma^d T_d
\]
in $\cT$ with the following properties: Each $T_i$ is in $\add T$, and applying the functor $\cT( T,- )$ gives a complex
\[
  \cT( T,T_d ) \rightarrow \cdots \rightarrow \cT( T,T_0 ) \rightarrow \cT( T,X ) \rightarrow 0
\]
which is the start of the augmented minimal projective resolution of $\cT( T,X )$.
\end{Lemma}

\begin{proof}
Given $X$, there exists a $( d+2 )$-angle 
\[
  \Sigma^{ -d }X \rightarrow T_d \rightarrow \cdots \rightarrow T_0 \rightarrow X
\]
with each $T_i$ in $\add T$ by Definition \ref{def:clusterTilting}.  Since $X$ has no non-zero direct summands in $\add \Sigma^d T$, the first morphism in the $( d+2 )$-angle is in the radical of $\cT$.  By dropping trivial summands of the form $T' \xrightarrow{ \cong } T'$, we can assume that so are the other morphisms except the last morphism.

By \cite[prop.\ 2.5(a)]{GKO}, applying the functor $\cT( T,- )$ gives an exact sequence
\[
  \cT( T,\Sigma^{ -d }X ) \rightarrow \cT( T,T_d ) \rightarrow \cdots \rightarrow \cT( T,T_0 ) \rightarrow \cT( T,X ) \rightarrow \cT( T,\Sigma^d T_d ) = 0.
\]
By Theorem \ref{thm:JJ1main}, applying the functor $\cT( T,- )$ is, up to isomorphism, just to apply a quotient functor, and this preserves radical morphisms.  So in the exact sequence each morphism, except possibly $\cT( T,T_0 ) \rightarrow \cT( T,X )$, is in the radical of $\mod \Gamma$. This proves the claim of the lemma.
\end{proof}

\begin{Lemma}\label{lem:ARtranslation}
If $X \in \cT$ has no non-zero direct summands in $\add \Sigma^d T$, then there is a natural isomorphism
\[
  \tau_d\HomT(T,X)\cong \HomT(T, \Sigma^dX).
\]
\end{Lemma}
\begin{proof}
As \(X\) has no non-zero direct summands in $\add \Sigma^d T$, we can consider the $( d+2 )$-angle from Lemma \ref{lem:resolution}.  Apply
\(\HomT(T,-)\) to get the following part of an augmented minimal projective resolution in \(\mod \Gamma\):
\[\HomT(T,T_d)\rightarrow \cdots \rightarrow \HomT(T,T_0)\rightarrow \HomT(T,X)\rightarrow 0.\]
Using the Nakayama functor and Lemma \ref{lem:Nakayama} we get the following commutative diagram.

\begin{center}
\begin{tikzpicture}[xscale=3, yscale=2]
\node[anchor=east] (0) at (-.5,1) {0};
\node (X) at (.5,1) {\(\tau_d\HomT(T,X)\)};
\node (Td) at (2,1) {\(\nu_\Gamma \HomT(T,T_d)\)};
\node (d1) at (3,1) {\(\cdots\)};
\node (T0) at (4,1) {\(\nu_\Gamma\HomT(T,T_0)\)};

\node[anchor=east] (SST0) at (-.5,0) {\(0\)};
\node (Y) at (.5,0) {\(\HomT(T,\Sigma^d X)\)};
\node (STd) at (2,0) {\(\HomT(T,\Sigma^{ 2d } T_d)\)};
\node (d0) at (3,0) {\(\cdots\)};
\node (ST0) at (4,0) {\(\HomT(T,\Sigma^{ 2d } T_0)\)};

\draw[->]	 	(0)--(X);
\draw[->]		(X)--(Td);
\draw[->]		(Td)--(d1);
\draw[->]		(d1)--(T0);

\draw[->]	 	(SST0)--(Y);
\draw[->]		(Y)--(STd);
\draw[->]		(STd)--(d0);
\draw[->]		(d0)--(ST0);

\draw[->] (Td)-- node[sloped, above]{\(\sim\)}(STd);
\draw[->] (T0)-- node[sloped, above]{\(\sim\)}(ST0);
\end{tikzpicture}
\end{center}
The top sequence is exact by the definition of $\tau_d$, see \cite[sec.\ 1.4.1]{Iyama2}.  The bottom sequence is exact because it is obtained by applying $\Hom_{ \cT }( T,- )$ to a $(d+2)$-angle in $\cT$, see \cite[prop.\ 2.5(a)]{GKO}.  The first term of the bottom sequence is actually $\cT( T,\Sigma^d T_0 )$, but this is zero.  Since we have $d \geq 1$, the diagram implies
\[
  \tau_d\HomT(T,X) \cong \HomT(T,\Sigma^dX).
\]
\end{proof}

We write \([\add T]( X,Y ) = \{\, f \in \HomT(X,Y) \mid f \text{ factors through an object of } \add T \,\}.\)

\begin{Lemma}
\label{lem:SerreHom}
There is a natural isomorphism
\[
  \dual[\add T](X,Y)\cong\Hom_{\T/\add\Sigma^d T}(\overline Y,\overline {\Sigma^{ 2d }X})\]
for $X,Y \in \cT$.
\end{Lemma}
\begin{proof}
Pick a \((d+2)\)-angle in \(\T\):
\[T_d\rightarrow \ldots\rightarrow T_0\rightarrow Y\rightarrow \Sigma^d T_d,\] with \(T_i\in\add T\). Use \(\HomT(X,-)\) to obtain the morphism \(\Psi :\HomT (X, T_0)\rightarrow \HomT(X,Y)\). This is a homomorphism of $k$-vector spaces, hence we can talk about the image of \(\Psi\). We first note that any morphism \(f\) in the image of \(\Psi\) must factor through \(\add T\). Now suppose \(f\in\HomT(X,Y)\) factors through \(T'\in\add T\). We have the following commutative diagram, where the lower row is a part of the \((d+2)\)-angle above:
\begin{center}
\begin{tikzpicture}[scale=2]
\node (d0) at (-1, 0) {\(\cdots\)};
\node (T0) at (0,0) {\(T_0\)};
\node (Y) at (1,0) {\(Y\)};
\node (ST) at (2,0) {\(\Sigma^d T_d\).};

\node (d1) at (-1, 1) {\(\cdots\)};
\node (Ta) at (0,1) {\(T'\)};
\node (Tb) at (1,1) {\(T'\)};
\node (0) at (2,1) {\(0\)};

\node (X) at (1,2) {\(X\)};

\draw[->] (d1)--(Ta);
\draw[->] (Ta)-- node[above, near end]{\(1_{T'}\)}(Tb);
\draw[->] (Tb)--(0);
\draw[->] (d0)--(T0);
\draw[->] (T0)--(Y);
\draw[->] (Y)--(ST);

\draw[->] (X)--(Tb);
\draw[->] (X)--(Ta);
\draw[->] (Tb)--(Y);
\draw[->] (0)--(ST);
\draw[->] (X) to[bend left] node[right, near start]{\(f\)} (Y);

\draw[->, dashed] (Ta)--(T0);
\end{tikzpicture}
\end{center}
The dashed arrow exists by completing the commutative square to a morphism of \((d+2)\)-angles.
We conclude that \(f\in\im\Psi\). Hence \[\im \Psi=[\add T](X,Y).\]

We now return to the long exact sequence
\[
  \cdots \rightarrow \HomT(X,T_0)\xrightarrow{\Psi} \HomT(X,Y)\rightarrow \HomT(X,\Sigma^dT_d)\rightarrow \cdots.
\]
Using the duality functor \(\dual\) and Serre duality we get the following diagram with exact rows:
\begin{center}
\begin{tikzpicture}[xscale=6, yscale=2]
\node (DTd) at (0,1) {\(\dual\!\HomT(X,\Sigma^dT_d)\)};
\node (DY) at (1,1) {\(\dual\!\HomT(X,Y)\)};
\node (DT0) at (2,1) {\(\dual\!\HomT(X,T_0)\)};

\node(STd) at (0,0) {\(\HomT(\Sigma^dT_d,\Sigma^{ 2d }X)\)};
\node(SY) at (1,0) {\(\HomT(Y,\Sigma^{ 2d }X)\)};
\node(ST0) at (2,0) {\(\HomT(T_0,\Sigma^{ 2d }X)\)};

\draw[->] (DTd)--(DY);
\draw[->] (DY)-- node[above]{\(\scriptstyle{\dual\!\Psi}\)}(DT0);
\draw[->] (STd)--node[above]{$\scriptstyle{\alpha'}$}(SY);
\draw[->] (SY)--node[above]{\(\scriptstyle{\beta'}\)}(ST0);

\draw[->] (DTd)-- node[sloped, above]{\(\sim\)}(STd);
\draw[->] (DY)-- node[sloped, above]{\(\sim\)}(SY);
\draw[->] (DT0)-- node[sloped, above]{\(\sim\)}(ST0);

\node (Im1) at (.5, -1) {\([\add \Sigma^d T](Y,\Sigma^{ 2d }X)\)};
\node (Im2) at (1.5, -1) {\(\HomT(Y,\Sigma^{ 2d }X)/[\add \Sigma^d T](Y,\Sigma^{ 2d }X)\)};

\draw[->] (STd)--(Im1);
\draw[->] (Im1)-- node[above]{\(\scriptstyle{\alpha}\)} (SY);
\draw[->] (SY)-- node[above]{\(\scriptstyle{\beta}\)} (Im2);
\draw[->] (Im2)--(ST0);
\end{tikzpicture}
\end{center}
Analogous to the above discussion, the space \([\add \Sigma^d T](Y,\Sigma^{ 2d }X)\) is the image of the map $\alpha'$. Hence \(\alpha\) is the kernel of \(\beta'\) and \(\dual\!\Psi\) (by isomorphism).  The morphism \(\beta\) is by definition the cokernel of  \(\alpha\), and \(\HomT(Y,\Sigma^{ 2d }X)/[\add \Sigma^d T](Y,\Sigma^{ 2d }X)\)  is thus the image of \(\dual\!\Psi\). Thus we have 
\[
  \dual[\add T](X,Y)
  \cong \dual\im\Psi 
  \cong \im \dual\!\Psi
  \cong \HomT(Y,\Sigma^{ 2d }X)/[\add \Sigma^d T](Y,\Sigma^{ 2d }X)
  \cong\Hom_{\T/\add\Sigma^d T}(\overline Y,\overline{\Sigma^{ 2d }X}).
\]
\end{proof}

\begin{Lemma}\label{lem:ARformulaQuotient}
Suppose \(X,Y\in \T\).Then we have a short exact sequence
\[
0
\rightarrow \dual\!\Hom_{\T/\add\Sigma^d T}(\overline Y,\overline{\Sigma ^d X})
\rightarrow \Ext^d_\T(X,Y)
\rightarrow \Hom_{\T/\add\Sigma^d T}(\overline X,\overline{\Sigma ^d Y})
\rightarrow 0.
\]
 \end{Lemma}
\begin{proof}
By the definition of the quotient functor we have a short exact sequence
\[0
\rightarrow [\add\Sigma^d T](X,\Sigma^d Y)
\rightarrow\HomT(X,\Sigma^dY)
\rightarrow \Hom_{\T/\add\Sigma^d T}(\overline X,\overline{\Sigma ^d Y})
\rightarrow 0.\]

We have \([\add\Sigma^d T](X,\Sigma^d Y)\cong[\add T](\Sigma^{-d}X, Y)\). By Lemma \ref{lem:SerreHom} we have 
\[
  [\add T](\Sigma^{-d}X, Y) \cong \dual\!\Hom_{\T/\add\Sigma^d T}(\overline Y,\overline{\Sigma^{2d}\Sigma ^{-d} X}) \cong \dual\!\Hom_{\T/\add\Sigma^d T}(\overline Y,\overline{\Sigma ^d X}).
\]

We also know that \(\HomT(X,\Sigma^dY)\cong \Ext^d_\T(X,Y)\), so the conclusion follows.
\end{proof}

\begin{Lemma}
\label{lem:ARformulaModule}
Suppose \(X,Y\in\T\) have no non-zero direct summands in \(\add \Sigma^d T\). Then we have a short exact sequence
\[
0\rightarrow
\dual\!\Hom_\Gamma \big( \HomT(T,Y),\tau_d\HomT(T,X) \big)
\rightarrow \Ext_\T^d (X,Y) \rightarrow 
\Hom_\Gamma \big( \HomT(T,X),\tau_d\HomT(T,Y) \big)
\rightarrow 0.
\]
\end{Lemma}
\begin{proof}
Consider the short exact sequence from Lemma \ref{lem:ARformulaQuotient}.
By Theorem \ref{thm:JJ1main} we know that
\[
  \dual\!\Hom_{\T/\add\Sigma^d T} ( \overline Y,\overline{\Sigma ^d X} )\cong 
\dual\!\Hom_\Gamma \big( \HomT(T,Y),\HomT(T,\Sigma^dX) \big).
\]
Applying Lemma \ref{lem:ARtranslation} we have 
\[
\dual\!\Hom_\Gamma \big( \HomT(T,Y),\HomT(T,\Sigma^dX) \big) \cong \dual\!\Hom_\Gamma \big( \HomT(T,Y),\tau_d\HomT(T,X) \big).
\]

Similarly we can show \(\Hom_{\T/\add\Sigma^d T}(\overline X,\overline{\Sigma ^d Y})\cong 
\Hom_\Gamma \big( \HomT(T,X),\tau_d\HomT(T,Y) \big)\).
\end{proof}

The map defined next will eventually induce the equivalence of Theorem B.

\begin{Definition}
\label{def:decomposition}
For each $X \in \cT$, pick an isomorphism \(X \cong X'\oplus X''\) such that \(X'\) has no non-zero direct summands in \(\add \Sigma^d T\) and \(X''\in\add\Sigma^dT\).  Let
\[
  \Delta( X ) = \big( \HomT( T,X' ),\HomT( T,\Sigma^{-d}X'' ) \big).
\]  
This is a pair of $\Gamma$-modules where $\HomT( T,X' )$ is in $\cD$ and \(\HomT(T,\Sigma^{-d}X'')\) is in $\proj \Gamma$.
\end{Definition}

\begin{Proposition}
\label{pro:dimFormula}
Given \(X,Y\in \T\), set $( M,P ) = \Delta( X )$ and $( N,Q ) = \Delta( Y )$, where $\Delta$ is the map in Definition \ref{def:decomposition}.  Then
\begin{align*}
\dim_k \Ext^d_\T(X,Y) =
&\dim_k \Hom_\Gamma ( M,\tau_d N )
+\dim_k \Hom_\Gamma ( N,\tau_d M ) \\
&+\dim_k \Hom_\Gamma ( P,N )
+\dim_k \Hom_\Gamma ( Q,M ).
\end{align*}
\end{Proposition}

\begin{proof}
By additivity of \(\Ext\) we have
\begin{align*}
 \Ext^d_\T(X,Y)&\cong\Ext^d_\T(X'\oplus X'',Y'\oplus Y'')\\
&\cong\Ext^d_\T(X',Y')\oplus\Ext^d_\T(X',Y'')\oplus\Ext^d_\T(X'',Y')\oplus\Ext^d_\T(X'',Y'').
\end{align*}
As \(T\) is $d$-rigid, we see that \(\Ext^d_\T(X'',Y'')=0\), and hence we have 
\begin{equation}
\label{equ:dimFormula1}
  \dim\Ext^d_\T(X,Y)=\dim\Ext^d_\T(X',Y')+\dim\Ext^d_\T(X',Y'') + \dim\Ext^d_\T(X'',Y').
\end{equation}
From Lemma \ref{lem:ARformulaModule} we have the short exact sequence:
\[
0\rightarrow
\dual\!\Hom_\Gamma \big( \HomT(T,Y'),\tau_d\HomT(T,X') \big)
\rightarrow \Ext_\T^d (X',Y') \rightarrow 
\Hom_\Gamma \big( \HomT(T,X'),\tau_d\HomT(T,Y') \big)
\rightarrow 0,
\]
which means that 
\begin{align}
\nonumber
  \dim \Ext_\T^d (X',Y')
  & = \dim_k \Hom_\Gamma \big( \HomT(T,X'),\tau_d\HomT(T, Y') \big) + \dim_k \Hom_\Gamma \big(\HomT(T,Y'),\tau_d\HomT(T, X') \big) \\
\label{equ:dimFormula2}
  & = \dim_k \Hom_{ \Gamma }( M,\tau_d N ) + \dim_k \Hom_{ \Gamma }( N,\tau_d M ).
\end{align}
We see that 
\[
\Ext_\T^d (X'',Y')\cong\HomT(X'',\Sigma^dY')\cong \HomT(\Sigma^{-d}X'',Y')\cong \Hom_\Gamma \big( \HomT(T,\Sigma^{-d}X''),\HomT(T, Y') \big) \cong \Hom_{ \Gamma }( P,N ).
\]
The third isomorphism follows from \cite[Lemma 2.2(i)]{JJ} and the fact that  \(\Sigma^{-d}X''\in\add T\).
Similarly, 
\[
  \Ext_\T^d (X',Y'')
  \cong \dual\! \Ext_\T^d (Y'',X')
  \cong \dual\! \Hom_\Gamma ( Q,M ).
\]
Thus we have
\begin{align}
\label{equ:dimFormula3}
  \dim \Ext_\T^d (X'',Y')&=\dim_k \Hom_\Gamma ( P,N ) \\
\label{equ:dimFormula4}
  \dim \Ext_\T^d (X',Y'')&=\dim_k \Hom_\Gamma ( Q,M ).
\end{align}
Substituting \eqref{equ:dimFormula2}, \eqref{equ:dimFormula3}, and \eqref{equ:dimFormula4} into \eqref{equ:dimFormula1} gives the result.
\end{proof}

As a consequence we have:

\begin{Corollary}
\label{cor:dimFormula_consequence}
Given \(X,Y\in \T\), set $( M,P ) = \Delta( X )$ and $( N,Q ) = \Delta( Y )$.  Then
\[
  \Ext^d_\T(X,Y) = 0
  \Leftrightarrow
  \Hom_\Gamma ( M,\tau_d N ) = \Hom_\Gamma ( N,\tau_d M ) =\Hom_\Gamma ( P,N ) = \Hom_\Gamma ( Q,M ) = 0.
\]
\end{Corollary}

\section{Proof of Theorem B}
\label{sec:tau-TO}

The following results use the map $\Delta$ from Definition \ref{def:decomposition}.

\begin{Lemma}
\label{lem:add}
Given \(X,Y\in \T\), set $( M,P ) = \Delta( X )$ and $( N,Q ) = \Delta( Y )$.  Then $Y \in \add X$ if and only if $N \in \add M \mbox{ and } Q \in \add P$.
\end{Lemma}

\begin{proof}
Let $X \cong X' \oplus X''$ be the decomposition from Definition \ref{def:decomposition}, where $X'$ has no non-zero direct summands from $\add \Sigma^d T$ while $X''$ is in $\add \Sigma^d T$.  We have $( M,P ) = \big( \cT( T,X' ),\cT( T,\Sigma^{ -d }X'' ) \big)$.  Similarly, $( N,Q ) = \big( \cT( T,Y' ),\cT( T,\Sigma^{ -d }Y'' ) \big)$.

The condition $Q \in \add P$ is equivalent to $Y'' \in \add X''$ by the \(\add\)-\(\proj\)-correspondence, (see Remark \ref{rem:add-proj}).  The condition $N \in \add M$ is equivalent to $Y' \in \add X'$ by Theorem \ref{thm:JJ1main} because \(X', Y'\) have no non-zero direct summands in \(\add \Sigma^{d}T\). The result follows.
\end{proof}

\begin{Lemma}
\label{lem:bijection}
The category \(\cT\) is skeletally small. 
The map $\Delta$ induces a bijection
\begin{equation}
\label{equ:bijection}
  \delta: \iso \cT \rightarrow \iso \cD \times \iso \proj \Gamma,
\end{equation}  
where $\iso$ denotes the set of isomorphism classes of a skeletally small category.
\end{Lemma}

\begin{proof}
Let $\Iso$ denote the class of isomorphisms of a category. For a skeletally small category $\cC$ we have that $\Iso \cC=\iso \cC$. Note that since a module category over a ring is skeletally small, we have that $\cD, \proj\Gamma \subseteq \mod \Gamma$ are skeletally small.

It is clear that $\Delta$ induces a well-defined map of the form 
  \[\delta': \Iso \cT \rightarrow \iso \cD \times \iso \proj \Gamma.\]

To see that $\delta'$ is injective, argue like the proof of Lemma \ref{lem:add}, replacing membership of $\add$ with isomorphism. 

It follows that $\cT$ is skeletally small. We can thus replace $\delta'$ with the map $\delta$ from (\ref{equ:bijection}).

To see that $\delta$ is surjective, let $( M,P )$ be a pair with $M \in \cD$ and $P \in \proj \Gamma$.  By Theorem \ref{thm:JJ1main} there is an object $X' \in \cT$ with no non-zero direct summands in $\add \Sigma^d T$ such that $M \cong \cT( T,X' )$.  By the add-proj correspondence, see Remark \ref{rem:add-proj}, there is an object $X'' \in \add \Sigma^d T$ such that $P \cong \cT( T,\Sigma^{ -d }X'' )$.  Setting $X = X' \oplus X''$ gives $( M,P ) \cong \Delta( X )$.
\end{proof}

\begin{Lemma}
\label{lem:main2_ii_a}
If $X \in \cT$ is $d$-self-perpendicular, then $( M,P ) = \Delta( X )$ is a maximal $\tau_d$-rigid pair.
\end{Lemma}

\begin{proof}
Let $N \in \cD$ and $Q \in \proj \Gamma$ be given.  By Lemma \ref{lem:bijection}, there is an object $Y \in \cT$ such that $( N,Q ) \cong \Delta( Y )$.  Then
\begin{align*}
  & N \in \add M \mbox{ and } Q \in \add P \\
  & \; \Leftrightarrow Y \in \add X \\
  & \; \Leftrightarrow \Ext^d_{ \cT }( X,Y ) = 0 \\
  & \; \Leftrightarrow \Hom_\Gamma ( M,\tau_d N ) = \Hom_\Gamma ( N,\tau_d M ) =\Hom_\Gamma ( P,N ) = \Hom_\Gamma ( Q,M ) = 0,
\end{align*}
where the equivalences, respectively, are by Lemma \ref{lem:add}, Definition \ref{def:self-orthogonal}, and Corollary \ref{cor:dimFormula_consequence}.

The conditions of Definition \ref{def:maximal_tau_d-rigid_pair} are recovered by setting \(Q=0\) respectively \(N=0\) 
\end{proof}

\begin{Lemma}
\label{lem:mains2_ii_b}
Let $X \in \cT$ be given.  If $( M,P ) = \Delta( X )$ is a maximal $\tau_d$-rigid pair, then $X$ is $d$-self-perpendicular.  
\end{Lemma}

\begin{proof}
Let $Y \in \cT$ be given and set $( N,Q ) \cong \Delta( Y )$.  Then
\begin{align*}
  & \Ext^d_{ \cT }( X,Y ) = 0 \\
  & \; \Leftrightarrow \Hom_\Gamma ( M,\tau_d N ) = \Hom_\Gamma ( N,\tau_d M ) =\Hom_\Gamma ( P,N ) = \Hom_\Gamma ( Q,M ) = 0 \\
  & \; \Leftrightarrow N \in \add M \mbox{ and } Q \in \add P \\
  & \; \Leftrightarrow Y \in \add X,
\end{align*}
where the equivalences, respectively, are by Corollary \ref{cor:dimFormula_consequence}, Definition \ref{def:maximal_tau_d-rigid_pair}, and Lemma \ref{lem:add}.
\end{proof}

\begin{Theorem}
\label{thm:main2}
Recall that the map $\Delta$ from Definition \ref{def:decomposition} induces the bijection $\delta: \iso \cT \rightarrow \iso \cD \times \iso \proj \Gamma$ from Lemma \ref{lem:bijection}.
\begin{enumerate}
\setlength\itemsep{4pt}

  \item  $\delta$ restricts to a bijection
\[  
  \left\{
    \begin{array}{cc}
      \mbox{isomorphism classes of } \\
      \mbox{$d$-rigid objects in $\cT$}
    \end{array}
  \right\}
  \rightarrow
  \left\{
    \begin{array}{cc}
      \mbox{isomorphism classes of} \\
      \mbox{$\tau_d$-rigid pairs in $\cD$}       
    \end{array}
  \right\}.
\]
  
  \item  $\delta$ restricts further to a bijection
\[  
  \left\{
    \begin{array}{cc}
      \mbox{isomorphism classes of } \\
      \mbox{$d$-self-perpendicular objects in $\cT$}
    \end{array}
  \right\}
  \rightarrow
  \left\{
    \begin{array}{cc}
      \mbox{isomorphism classes of} \\
      \mbox{maximal $\tau_d$-rigid pairs in $\cD$}       
    \end{array}
  \right\}.
\]

\end{enumerate}
\end{Theorem}

\begin{proof}
(i):  Consider $X \in \cT$ and set $( M,P ) = \Delta( X )$.  Then
\[
  \Ext^d_{ \cT }( X,X ) = 0
  \Leftrightarrow
  \Hom_{ \Gamma }( M,\tau_d M ) = 0
  \mbox{ and }
  \Hom_{ \Gamma }( P,M ) = 0
\]
by Corollary \ref{cor:dimFormula_consequence}, so the result follows.

(ii): See Lemmas \ref{lem:main2_ii_a} and \ref{lem:mains2_ii_b}.
\end{proof}

{\it Proof }(of Theorem B from the introduction).
Combine Theorems \ref{thm:main2}(ii) and \ref{thm:cluster_tilting_v_maximal_rigid}(ii).
\hfill $\Box$

\section{An example}
\label{sec:example}

In this section we let $d = 3$ and $\cT = \cO_{ A_2^3 }$.  This is the $5$-angulated (higher) cluster category of type $A_2$, see \cite[def.\ 5.2, sec.\ 6, and sec.\ 8]{OT}.  The indecomposable objects can be identified with the elements of the set 
\[
  {}^{ \circlearrowleft }\mbox{\bf I}^3_9
  =
  \{\, 1357,1358,1368,1468,2468,2469,2479,2579,3579 \,\},
\]
see \cite[sec.\ 8]{OT}.  The AR quiver of $\cT$ is shown in Figure \ref{fig:AR_quiver}.
\begin{figure}
\begin{tikzpicture}[scale=3]
  \node at (0:1.0){$1357$};
  \draw[->] (10:1.0) arc (10:30:1.0);
  \node at (40:1.0){$1358$};
  \draw[->] (50:1.0) arc (50:68:1.0);
  \node at (80:1.0){$1368$};
  \draw[->] (91:1.0) arc (91:108:1.0);
  \node at (120:1.0){$1468$};
  \draw[->] (131:1.0) arc (131:150:1.0);
  \node at (160:1.0){$2468$};
  \draw[->] (170:1.0) arc (170:190:1.0);
  \node at (200:1.0){$2469$};
  \draw[->] (210:1.0) arc (210:227:1.0);
  \node at (240:1.0){$2479$};
  \draw[->] (251:1.0) arc (251:268:1.0);
  \node at (280:1.0){$2579$};
  \draw[->] (295:1.0) arc (295:310:1.0);
  \node at (320:1.0){$3579$};
  \draw[->] (330:1.0) arc (330:351:1.0);
\end{tikzpicture}
\caption{The AR quiver of the $5$-angulated category $\cT$.}
\label{fig:AR_quiver}
\end{figure}
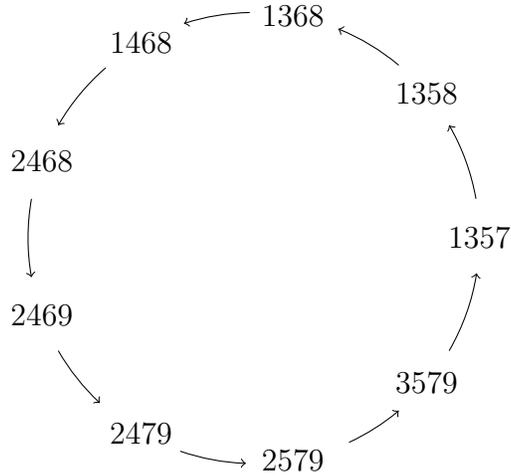
By \cite[thm.\ 5.5 and sec.\ 8]{OT}, the object 
\[
  T = 1357 \oplus 1358 \oplus 1368 \oplus 1468
\]  
is Oppermann--Thomas cluster tilting.  

If $X,Y \in \cT$ are indecomposable objects, then
\[
  \cT( X,Y )
  =
  \left\{
    \begin{array}{cl}
      k & \mbox{ if $Y$ is $X$ or its immediate successor in the AR quiver, }\\
      0 & \mbox{ otherwise, }
    \end{array}
  \right.
\]
see \cite[prop.\ 6.1 and def.\ 6.9]{OT}.  It follows that $\Gamma = \End_{ \cT }( T ) = kQ / I$, where
\[
  Q = 1 \rightarrow 2 \rightarrow 3 \rightarrow 4
\]
and $I$ is the ideal generated by all compositions of two consecutive arrows.  The action of the functor $\cT( T,- ) : \cT \rightarrow \mod \Gamma$ on indecomposable objects is shown in Figure \ref{fig:quotient_functor}, where $P( q )$ and $I( q )$ denote the indecomposable projective and injective modules associated to the vertex $q \in Q$.  
\begin{figure}
\begingroup
\renewcommand\arraystretch{1.5}
\begin{tabular}{c|ccccccccc}
  $X$ & 1357 & 1358 & 1368 & 1468 & 2468 & 2469 & 2479 & 2579 & 3579 \\\cline{1-10}
  $\cT( T,X )$ & $P(4)$ & $P(3)$ & $P(2)$ & $P(1)$ & $I(1)$ & $0$ & $0$ & $0$ & $0$
\end{tabular}
\endgroup
\caption{The action of the functor $\cT( T,- ) : \cT \rightarrow \mod \Gamma$.}
\label{fig:quotient_functor}
\end{figure}
Note that the essential image of $\cT( T,- )$ is
\[
  \cD = \add \{\, P(4),P(3),P(2),P(1),I(1) \,\}.
\]  
This is a $3$-cluster tilting subcategory of $\mod \Gamma$ and hence it is $3$-abelian.

The $3$-suspension functor $\Sigma^3$ acts on the AR quiver by moving four steps clockwise.  Combined with our knowledge of $\Hom$, this shows that if $X$ is a fixed indecomposable object in $\cT$, then the indecomposable objects $Y$ with $\Ext^3_{ \cT }( X,Y ) \neq 0$ are precisely the two objects furthest from $X$ in the AR quiver, see Figure \ref{fig:forbidden}.
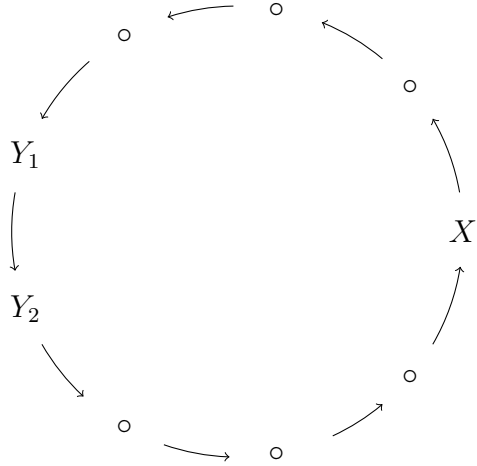
\begin{figure}
\begin{tikzpicture}[scale=3]
  \node at (0:1.0){$X$};
  \draw[->] (10:1.0) arc (10:30:1.0);
  \node at (40:1.0){$\circ$};
  \draw[->] (50:1.0) arc (50:68:1.0);
  \node at (80:1.0){$\circ$};
  \draw[->] (91:1.0) arc (91:108:1.0);
  \node at (120:1.0){$\circ$};
  \draw[->] (131:1.0) arc (131:150:1.0);
  \node at (160:1.0){$Y_1$};
  \draw[->] (170:1.0) arc (170:190:1.0);
  \node at (200:1.0){$Y_2$};
  \draw[->] (210:1.0) arc (210:227:1.0);
  \node at (240:1.0){$\circ$};
  \draw[->] (251:1.0) arc (251:268:1.0);
  \node at (280:1.0){$\circ$};
  \draw[->] (295:1.0) arc (295:310:1.0);
  \node at (320:1.0){$\circ$};
  \draw[->] (330:1.0) arc (330:351:1.0);
\end{tikzpicture}
\caption{The functor $\Ext^3_{ \cT }( X,- )$ is non-zero on $Y_1$ and $Y_2$.  It is zero on every other indecomposable object.}
\label{fig:forbidden}
\end{figure}

Based on this, we can compute all basic $3$-self-perpendicular objects in $\cT$, and by Proposition \ref{pro:indecomposables_are_rigid} they coincide with the basic maximal $3$-rigid objects in $\cT$.  For each such object $X$, there is a maximal $\tau_3$-rigid pair $\Delta( X ) = \big( \cT( T,X' ),\cT( T,\Sigma^{ -3 }X'' ) \big)$ by Theorem B.  See Figure \ref{fig:example}.
\begin{figure}
\begingroup
\renewcommand\arraystretch{1.5}
\begin{tabular}{c|c}
  Maximal $3$-rigid object $X$ & Maximal $\tau_3$-rigid pair $\Delta( X )$ \\ \cline{1-2}
  $1357 \oplus 1358 \oplus 1368 \oplus 1468$ & $( \Gamma,0 )$ \\
  $1358 \oplus 1368 \oplus 1468 \oplus 2468$ & $( \dual\!\Gamma,0 )$ \\
  $1368 \oplus 1468 \oplus 2468 \oplus 2469$ & $\big( P(2) \oplus P(1) \oplus I(1),P(4) \big)$ \\
  $1468 \oplus 2468 \oplus 2469 \oplus 2479$ & $\big( P(1) \oplus I(1),P(4) \oplus P(3) \big)$ \\
  $2468 \oplus 2469 \oplus 2479 \oplus 2579$ & $\big( I(1),P(4) \oplus P(3) \oplus P(2) \big)$ \\
  $2469 \oplus 2479 \oplus 2579 \oplus 3579$ & $( 0,\Gamma )$ \\
  $2479 \oplus 2579 \oplus 3579 \oplus 1357$ & $\big( P(4),P(3) \oplus P(2) \oplus P(1) \big)$ \\
  $2579 \oplus 3579 \oplus 1357 \oplus 1358$ & $\big( P(4) \oplus P(3),P(2) \oplus P(1) \big)$ \\
  $3579 \oplus 1357 \oplus 1358 \oplus 1368$ & $\big( P(4) \oplus P(3) \oplus P(2),P(1) \big) $ \\
  $1357 \oplus 1468 \oplus 2479            $ & $\big( P(4) \oplus P(1),P(3) \big)$ \\
  $1358 \oplus 2468 \oplus 2579            $ & $\big( P(3) \oplus I(1),P(2) \big)$ \\
  $1368 \oplus 2469 \oplus 3579            $ & $\big( P(2),P(4) \oplus P(1) \big)$
\end{tabular}
\endgroup
\caption{These are all the basic maximal $3$-rigid objects of $\cT$ and their corresponding maximal $\tau_3$-rigid pairs in $\cD$.} 
\label{fig:example}
\end{figure}
Note that the first nine objects in Figure \ref{fig:example} are Oppermann--Thomas cluster tilting, but the three last objects are not.

{\bf Acknowledgement.}
This work was supported by EPSRC grant EP/P016014/1 ``Higher Dimensional Homological Algebra''.  Karin M. Jacobsen is grateful for the hospitality of Newcastle University during her visit in October 2018.

\end{document}